\documentclass[12pt]{article}
\usepackage{amssymb}
\usepackage{amsfonts}
\usepackage{amsmath}
\usepackage[usenames]{color}
\usepackage{mathrsfs}
\usepackage{amsfonts}
\usepackage{amssymb,amsmath}
\usepackage{CJK}
\usepackage{cite}
\usepackage{cases}
\usepackage{amsthm}

\def\re#1{\par\hang\tex{#1}}

\def\cal{\mathcal}
\def\gg{\frak g}
\def\a{\alpha}
\def\b{\beta}
\def\d{\delta}
\def\D{\Delta}
\def\UU{{U}}

\def\LL{{\cal L}}
\def\SS{{\cal{S}}}
\def\g{\gamma}

\def\vp{\varepsilon}

\def\v{\varphi}

\def\sc{\scriptstyle}
\def\ssc{\scriptscriptstyle}

\def\cl{\centerline}

\def\rar{\rightarrow}

\def\vs{\vspace*}

\def\AA{{\cal A}}

\def\LL{{\cal L}}

\def\ni{\noindent}
\def\ptl{\partial}

\def\N{\mathbb{N}{\ssc\,}}
\def\Z{\mathbb{Z}{\ssc\,}}

\def\C{\mathbb{C}{\ssc\,}}
\def\C{{\cal K}}
\def\F{{\cal K}}
\def\QED{\hfill$\Box$} \numberwithin{equation}{section}
\newtheorem{theo}{Theorem}[section]
\newtheorem{conv}[theo]{Convention}

\newtheorem{rema}[theo]{Remark}
\newtheorem{lemm}[theo]{Lemma}
\newtheorem{prop}[theo]{Proposition}

\def\adddot{$\!\!\!${\bf.}\ \ }
\def\adddot{}

\usepackage{CJK}
\begin{document}
\begin{CJK*}{GBK}{song}

\begin{center}{{\large \bf Dual Lie bialgebra structures of the twisted \\ Heisenberg-Virasoro type}
\footnote{ Supported by NSF grant 11671056, 11431010, 11371278
 of China and NSF grant ZR2013AL013, ZR2014AL001 of Shandong
Province.}}
\end{center}

\cl{ Guang'ai Song$^{\,*}$, Yucai Su$^{\,\dag}$, Xiaoqing Yue$^{\,\dag}$}

\cl{\small $^{*\,}$College of Mathematics and Information Science, Shandong
Technology and \vs{-4pt}Business }

\cl{\small University, Yantai, Shandong 264005, China}

\cl{$^{\dag\,}$School of Mathematical Sciences, Tongji University, Shanghai 200092, China}

\cl{\small E-mails: gasong@sdibt.edu.cn, ycsu@tongji.edu.cn, xiaoqingyue@tongji.edu.cn}\vskip8pt

\noindent{\small{\bf Abstract:} In this paper, by studying the maximal good subspaces, we determine the dual Lie coalgebras of the centerless twisted Heisenberg-Virasoro algebra. Based on this, we
construct the dual Lie bialgebras structures of the twisted Heisenberg-Virasoro type. As by-products, four new infinite dimensional Lie algebras are obtained.
\vs{5pt}


}

\section{Introduction}\setcounter{section}{1}\setcounter{equation}{0}

The twisted Heisenberg-Virasoro algebra, first studied in \cite{A}, is an important algebra
structure, which has close relations with the Heisenberg algebra and the
Virasoro algebra.
In recent years, more and more attentions have been paid to this
algebra (see, e.g., \cite{E-D, B, F-S, L-C, L-W, L-Z, S-C}).
Let us first recall the definition here. The {\it twisted Heisenberg-Virasoro algebra} is a
Lie algebra with the underlining vector space $\cal{HV} = {\rm span}_{\C}\{L_m, I_n, C_L, C_I, C_{LI}\ |\ m,\, n \in \Z \}$ over an algebraically closed filed $\F$ of characteristic zero, subject to the following relations:
\begin{equation}\label{1.2}\begin{array}{l}
{[L_m, L_n] = (n-m)L_{m+n} + \d_{m+n,0}\frac{1}{12}(m^3-m)C_L,}\\
{ [I_m, I_n]= n\d_{m+n,0}C_I,}\\
{ [L_m, I_n]= nI_{m+n}+ \d_{m+n,0}C_{LI},}\\
{ [\LL, C_L ]= [\LL, C_I] = [\LL, C_{LI}]=0, \ \ \ m,\,n \in \Z.}\end{array}
\end{equation}
Obviously, the Virasoro algebra $\cal{V}= {\rm span}_{\C}\{L_m, C_L \,|\,m\in \Z \}$ and the Heisenberg algebra $\cal{H}={\rm span}_{\C}\{I_m, C_I\, |\, m\in \Z \}$
 are subalgebras of $\LL$.

In 
\cite{L-P}, the Lie bialgebra structures of the twisted Heisenberg-Virasoro type were investigated.
In the present paper, we will study the dual Lie bialgebra
structures of the twisted Heisenberg-Virasoro Lie bialgebra.
It is well-known that
the notion of Lie bialgebras was introduced by Drinfeld in 1983 \cite{D1,2} in connection with quantum groups.
Since Lie bialgebras as well as their quantizations provide important tools in searching for solutions of quantum
 Yang-Baxter equations and in producing new quantum groups, a number of papers on Lie bialgebras have appeared
 (e.g., \cite{1,1-2,D,G,M1,3,4,5,6,7,8,9,10,12,13,14}).
For instance, the structures of Witt and Virasoro type Lie bialgebras were presented in \cite{8,4}, and a
classification of this type Lie bialgebras was given in \cite{9}. All Lie bialgebra
structures on the Witt, the one-sided Witt, and the Virasoro algebras were shown to be triangular
coboundary, which can be obtained from their nonabelian two dimensional Lie subalgebras (cf.~\cite{8}).
For the generalized Witt type Lie bialgebras cases, the authors in \cite{10} obtained that all structures of Lie bialgebras on them are
coboundary triangular. Similar results also hold for some other kinds of Lie bialgebras (cf., e.g., \cite{13, 14}).

As stated in \cite{16,SS},
Lie bialgebra structures of coboundary triangular type may sound simple, but they are not
trivial. Indeed, there are many natural problems associated with them remain open. For example, even for  the (two-sided) Witt algebra and the Virasoro
algebra, a completely classification of coboundary triangular Lie
bialgebra structures on them is still open. Nevertheless, rather few is known on
representations of infinite dimensional Lie bialgebras. Therefore, it seems to us
that more attentions should be paid on this aspect. The authors of \cite{16,SS}  studied
dual Lie bialgebra structures of the (two-sided) Witt algebra, the Virasoro
algebra, the Poisson algebra and the loop and current-Virasoro
type algebras. As by-products, some new series of infinite dimensional
Lie algebras are obtained.
Studying dual Lie bialgebra structures can also
provide new approaches to investigate quantum groups, especially in
studying Lie bialgebras, and also help us to understand why we state that
the coboundary triangular Lie bialgebras are not trivial.
 We remind that the main problems occurring in the study of
dual Lie bialgebras are: (1) the determination
of the maximal good spaces, (2) the determination of the Lie algebra structures, especially (1).
 In the present paper, we have found an efficient way in tackling problem (1) (cf.~Theorem \ref{Th-3-1} and \eqref{MaxHVir}).

This paper proceeds as follows. Some definitions and preliminary
results are briefly recalled in Section 2. Then
structures of dual coalgebras of centerless twisted
Heisenberg-Virasoro algebra  are addressed. In Section 3, structures
of dual Lie bialgebras of Heisenberg-Virasoro algebra are
investigated. The main results of the present paper are summarized
in Theorems \ref{Th-3-1}, \ref{Th-4-1}, \ref{Th-4-2}, \ref{Th-4-3},
\ref{Th-4-4}, \ref{Th-4-5}.

\section{Definitions and preliminary
results}\setcounter{section}{2}\setcounter{equation}{0}

Let us briefly recall some notions on Lie bialgebras, for details, we
refer readers to, e.g., \cite{2,10,SS}. Throughout this paper, all vector spaces are
assumed to be over an algebraically closed field $\F$ of
characteristic zero, and as usual, $\N$ denotes the set
of nonnegative integers.

A triple $(\LL, [\cdot, \cdot], \d)$ is called {\it a Lie bialgebra}, if it satisfies the following  conditions:
\begin{enumerate}\parskip-1pt
\item[{\rm (1)}] $(\LL, [\cdot, \cdot])$ is a Lie algebra and
$(\LL, \d)$ is a Lie coalgebra\,;
\item[{\rm (2)}] 
$\d[x, y] = x \cdot \d(y) - y\cdot\d(x)$ for all $x,y \in \LL$,
\end{enumerate}
where $\d:\LL\to \LL\otimes\LL$ is a derivation and $ x \cdot (y\otimes z) = [x, y]\otimes z + y\otimes [x, z]$ for  $x, y, z \in \LL.$

A  Lie bialgebra $(\LL, [\cdot, \cdot],  \d)$ is {\it coboundary} if $\d$ is coboundary in the sense that there exists $r \in \LL \otimes \LL$ written as $r= \sum r^{[1]} \otimes r^{[2]} $, such that $\d(x) = x\cdot r$
for $x\in \LL$.
A coboundary Lie bialgebra  $(\LL, [\cdot, \cdot],  \d)$ is
{\it triangular} if $r$ satisfies  the following {\it classical Yang-Baxter equation} (CYBE):
\begin{equation}\label{CYBE}C(r) = [r_{12}, r_{13}] + [r_{12}, r_{23}] + [r_{13}, r_{23}]
 =0,\end{equation}
 \noindent where
$r_{12} = \sum r^{[1]} \otimes r^{[2]} \otimes 1,$ $r_{13}= \sum
r^{[1]} \otimes 1 \otimes r^{[2]},$ $r_{23} = \sum 1 \otimes r^{[1]}
\otimes r^{[2]}$ are elements in $ \UU(\LL)^{\otimes3}$ and $\UU(\LL)$ is the universal enveloping
algebra of $\LL$.

Two Lie bialgebras $(\LL,[\cdot,\cdot],\d)$ and
$(\LL',[\cdot,\cdot]',\d')$ are called {\it dually paired} if there
exists a nondegenerate bilinear form $\langle\cdot,\cdot\rangle$ on
$\LL'\times\LL$ $\big($extended uniquely to a bilinear form on
\mbox{$(\LL'\otimes\LL')\times(\LL \otimes\LL)$}$\big)$ such that their
bialgebra structures are related via
\begin{equation}\label{Dual-paired}
\langle[f, h ]', \xi\rangle = \langle f \otimes h, \d \xi \rangle ,\
\ \  \langle \d' f, \xi \otimes \eta \rangle = \langle f, [\xi,
\eta]\rangle \mbox{ \ for \ $f, h \in \LL',\ \xi, \eta \in \LL.$
}\end{equation}
 In particular, $\LL$ is called a {\it self-dual Lie bialgebra} if $\LL'=\LL$ as a vector space.

Note that a finite dimensional Lie bialgebra
$(\LL,[\cdot,\cdot],\d)$ is always self-dual as the linear dual
space $\LL ^*$ is naturally a Lie bialgebra by dualization and
there exists a vector space isomorphism $\LL\to\LL^*$ which pulls
back the bialgebra structure on $\LL^*$ to $\LL$ to obtain another
bialgebra structure on $\LL$ to make it to be self-dual. However,
infinite dimensional Lie bialgebras have sharp differences, as they are
not self-dual in general.

For convenience, we denote the Lie bracket of Lie algebra $\LL=(\LL ,
[\cdot, \cdot])$ by $\v$, i.e., $[\cdot, \cdot] =
\v:\LL\otimes\LL\to\LL$ and $\v^{\ast}:\LL^*\to(\LL\otimes\LL)^*$
to be the dual of $\v$.

A  subspace $U$ of $\LL
^*$ is called a {\it good subspace} if $\v^*(U) \subset U \otimes
U.$ It follows that $\LL ^{\circ}$ defined below
is also a good subspace
of $\LL^*$, which is obviously the {\it maximal good subspace} of $\LL^*$ \cite{M1}:
\begin{equation}\label{max-good}\LL ^{\circ} = \mbox{$\sum\limits_{U\in \Re}$} U,
\mbox{ \ where $\Re = \{ U\, |\, U {\rm\ is \ a \ good \ subspace \ of \
}\LL^*\}$.}\end{equation}
The notion of good subspaces of an associative algebra can be
defined analogously.
It is proved in  \cite{M1} that for
any good subspace $U$ of $\LL^*$, the pair $(U,
\v^{\ast}|_{U})$ is a Lie coalgebra. In particular, $(\LL^{\circ},
\v^*|_{\LL^{\circ}})$ is a Lie coalgebra.
%

 For any Lie algebra $\LL$, the dual space $\LL^*$ has a natural right $\LL$-module
 structure defined by $(f\cdot x)(y) = f([x,
 y])\mbox{  for  }f\in \LL^*,\ x, y\in \LL.$ We denote $f\cdot \LL  = {\rm span} \{ f\cdot x\, |\, x \in
 \LL \}$, the {\it space of translates} of $f$ by elements of $\LL$.

We summarize some results of \cite{1,1-2,D, G} as follows.
 \begin{prop}\adddot\label{Th111}
 Let $\LL $ be a Lie algebra.
 Then\begin{enumerate}\parskip-1pt\item[{\rm (1)}] $\LL^{\circ}=\{f\in\LL^*\,|\,f\cdot\LL\mbox{ is finite dimensional}{\sc\,}\}.$
\item[{\rm (2)}]
$\LL^{\circ}=(\v^*)^{-1}(\LL^*\otimes\LL^*)$, the preimage of
$\LL^*\otimes\LL^*$ in $\LL^*$.
\end{enumerate}
 \end{prop}
%

%
For an infinite dimensional Lie algebra $g$, there is no effective
approach to determine the good subspace $g^{\circ}$ of it. However, for an
associative commutative algebra $\cal{A}$, Sweedler \cite{12}
gave some approaches to determine $\cal{A^{\circ}}$. In the cases of
$\cal{A}= \F[x]$ and $\F[x^{\pm 1}]$, the maximal good subspaces of $\cal{
A}$ were determined (see \cite{5,6,7, 8, 10, 16, SS}).
Although the property of the good subspaces of an associative
commutative algebra has great difference with the Lie algebra case,
if a Lie algebra can be induced from an associative commutative
algebra, then the good subspaces of this  Lie algebra can be
determined through the associative commutative algebra case. Therefore, let us recall some results about associative commutative
algebra for later use.

Let $(\AA, \mu)$ be an associative commutative algebra over a field $\F$. By Proposition
\ref{Th111} and \cite{12}, we have $\AA^{\circ} = (\mu^*)^{-1}(\AA^* \otimes \AA^*)$. For $\ptl \in {\rm Der}(\AA)$, since
$$\mbox{$\ptl \mu =\mu( {\rm id} \otimes \ptl + \ptl \otimes {\rm id})$, \ $\mu^*\ptl^*(f) = ({\rm id} \otimes \ptl^* + \ptl^* \otimes {\rm id}
)\mu^*(f) \in \cal{A^*} \otimes \cal{A^*}$,}$$
where $f\in \AA^{\circ}$, it follows that
$\ptl^*(\cal{A^{\circ}}) \subset \cal{A^{\circ}}$.

In the case of $\AA= \F[x^{\pm 1}]$, let $\SS= \{\,x^n\, |\, n\in \Z
\}$ be the standard basis of $\AA$ and $\SS^{\prime}= \{\,\vp^n\, |\,
n\in \Z \}$ be the dual basis of $\SS$, i.e.,
$\vp^i (x^j) = \d_{i,j}
$ for $i, j \in  \Z.$ For
$f \in \AA^*$, $f$ can be expressed as $f=\sum_{j\in \Z} f_j
\vp^j.$
The structures of $\AA^{\circ}$ can be found as follows (see \cite{8, 16, SS, SS-2}).

\begin{lemm}\label{lem-3-2}Let $f=\sum_{j\in \Z} f_j \vp^j \in \AA^* $. Then
$f \in \AA^{\circ}$ if and only if there exist $r\in \N$ and $c_j \in \F$ for $j=1, 2, \cdots, r$
such that  $f_n = c_1 f_{n-1} + c_2f_{n-2} + \cdots + c_r f_{n-r}$ for all $n\in\Z.$
\end{lemm}

Let $(\cal{A}, \mu_1)$ and $(\cal{B}, \mu_2)$ be two  associative
commutative algebras, $\mu_1, \mu_2$ be the multiplications of
$\cal{A}$ and $\cal{B}$ respectively. Define the direct sum $\cal{L}
= \cal{A} \oplus \cal{B}$ with the multiplication $\mu = (\mu_1,
\mu_2)$, i.e. $\mu\big((a_1, b_1), (a_2, b_2)\big) = \big(\mu_1(a_1, a_2),
\mu_2(b_1, b_2)\big)$. Then $(\cal{L}, \mu)$ is also an associative
commutative algebra, and it is easy to prove the following lemma.

\begin{lemm}\label{lem-3-1} Let $(\cal{L}, \mu)$ be the above associative commutative
algebra, then 
$\cal{L^{\circ}} = ({\cal A}\oplus {\cal B})^{\circ}= {\cal A}^{\circ}
\oplus {\cal B}^{\circ}.$
\end{lemm}

\def\LL{{\overline{\cal{HV}}}{\ssc\,}}%
The twisted Heisenberg-Virasoro algebra \eqref{1.2} is the universal central extension of the algebra
with the underlining vector space $\LL ={\rm span}_{\C}\{L_m, I_n \,|\,m,n \in \Z \}$
and the brackets
\begin{equation}\label{1-1-1}\begin{array}{l}
{[L_m, L_n] = (n-m)L_{m+n},}\\
{ [I_m, I_n]= 0,}\\
{ [L_m, I_n]= nI_{m+n},\ \  m,\, n \in \Z.}\end{array}
\end{equation}
Thus it is called the {\it centerless
twisted Heisenberg-Virasoro algebra}.  This algebra has a polynomial
realization as follows. Let $\cal{A}= \C[x^{\pm 1}]$ be the
Laurent polynomial over complex field $\C$, with one variable $x$
 and the usual derivation $\ptl =\frac{d}{dx}$ of $\C[x^{\pm 1}]$. Denote the direct sum
of Laurent polynomial $\C[x^{\pm 1}]$ by $\LL =
\AA \oplus \AA =\C[x^{\pm 1}] \oplus \C[x^{\pm 1}]$. For $\big(f_1(x), f_2(x)\big),
\big(g_1(x), g_2(x)\big) \in \LL $, we define the multiplication $\mu = (\mu_1,
\mu_2)$ on $\cal{L}$ by
\begin{equation}\label{1-1-3}
\mu\Big(\big(f_1(x), f_2(x)\big), \big(g_1(x), g_2(x)\big)\Big)=\Big(\mu_1\big(f_1(x), g_1(x)\big),
\mu_2\big(f_2(x), g_2(x)\big)\Big),\end{equation}
 \noindent where $\mu_1$ is the multiplication of the first copy $\AA$ of $\LL$ and $\mu_2$ is the multiplication
of the second copy $\AA$ of $\LL$ in the usual way.
 Then $(\LL, \mu)$ is an associative commutative algebra.

Denote\begin{equation}\label{denotellll}
L_m = (x^{m+1}, 0),\ I_n = (0, x^n)\in\LL\mbox{ for }m, n \in \Z.\end{equation} Define $(\ptl, {\rm id})\big((x^k, x^l)\big) = \big(\ptl(x^k),
{\rm id}(x^l)\big)$ and some other similar operators on $\LL$ are defined in the same way.
Now let $\tau,{\rm id}$ and $T$ be operators on  $\LL$ and $\LL\otimes\LL$ such that $$\tau(x^m, x^n)= (x^n, x^m),\ \ {\rm id}(x^m, x^n)=(x^m, x^n),\ \
T\big((x^m, x^n), (x^k, x^l)\big)= \big((x^k, x^l), (x^m, x^n)\big).$$
Then the brackets  (\ref{1-1-1}) can be realized as follows.
\begin{equation}\label{1-1-2}
\aligned
\,[L_m, L_n] &=\mu\big(({\rm id}, 0) \otimes (\ptl, 0)- (\ptl, 0) \otimes ({\rm id},0)\big)\big((x^{m+1}, 0), (x^{n+1}, 0)\big)
\\&=(n-m)
(x^{m+n+1}, 0),
\\ [I_m, I_n] &=\mu\Big((0, {\rm id})\otimes (0, {\rm id})-\big((0, {\rm id})\otimes (0, {\rm id})\big)\cdot T\Big)\big((0, x^m), (0, x^n)\big)
\\&=0,
\\ [L_m, I_n]&=\mu\Big(\big((0, {\rm id})\otimes(0, \ptl)\big)\cdot(\tau, {\rm id})\Big)\big((x^{m+1},
0), (0, x^n)\big)\\&= n(0, x^{m+n}),\ \ \ \ m,\ n \in \Z.
\endaligned
\end{equation}
%
%
%
%
%
%
For the centerless twisted Heisenberg-Virasoro algebra defined by
(\ref{1-1-2}), there are two natural approaches to determine the Lie
coalgebras structures on some subspaces of $\LL^*$. One is
to converse the arrow of the Lie bracket $\v.$ That is, let
$\LL^{\circ}_{\v}$ be the maximal good subspace of $\LL^*$
under $\v^*: \LL^* \rar (\LL \otimes \LL)^*$,  which is
the dual multiplication of the Lie bracket  $\v: \LL \otimes
\LL \rar \LL$. Take $\v^{\circ}
=\v^*|_{{\LL^{\circ}_{\v}}}$, then we obtain $\v^{\circ}(
{\LL^{\circ}_{\v}}) \subset  {\LL^{\circ}_{\v}} \otimes
{\LL^{\circ}_{\v}}$ and $({\LL^{\circ}_{\v}}, \v^{\circ})$ is a
Lie coalgebra, which we call the dual Lie coalgebra of Lie
algebra $(\LL, \v)$.

Another approach is induced from the coassociative cocommutative
coalgebra $({\LL^{\circ}}, \mu^{\circ})$. 
Let $(\LL = \AA \oplus \AA, \mu)$ be
the associative commutative algebra defined by (\ref{1-1-3}) with $\mu:
\LL \otimes \LL \rar \LL. $ Then $\mu^*: \LL^* \rar (\LL \otimes
\LL)^*. $ Denote $\mu^{\circ} = \mu^* |_{\LL^{\circ}_{\mu}}$ and
$\mu^{\circ}:\LL_{\mu}^{\circ} \rar \LL_{\mu}^{\circ} \otimes
\LL_{\mu}^{\circ}.$ For $f \in \LL_{\mu}^{\circ},\ f=(f_1, f_2) \in
\LL_{\mu}^{\circ} = (\AA^{\circ} \oplus \AA^{\circ})$ (by Lemma
\ref{lem-3-1}), we have
\begin{equation}\label{eq-3-1}
\mu^{\circ}(f) =\mu^{\circ}((f_1, 0) + (0, f_2))=\sum\limits_{(f_1)}(f_1^{(1)}, 0) \otimes (f_1^{(2)}, 0)
+ \sum\limits_{(f_2)}(0, f_2^{(1)}) \otimes (0, f_2^{(2)}),
\end{equation}
\noindent where $ \mu_1^{\circ}(f_1) = \sum_{f_1}f_1^{(1)}
\otimes f_1^{(2)},\ \mu_2^{\circ} (f_2) =\sum_{f_2} f_2^{(1)}
\otimes f_2^{(2)}.$
%
%
Let $\ptl^{\circ}= \ptl^{\ast}|_{\AA^{\circ}}$, using
(\ref{1-1-2}), for $f = (f_1, f_2)=(f_1, 0)+ (0, f_2) \in
\LL^{\circ}_{\mu} = \AA^{\circ} \oplus \AA^{\circ} $, we obtain
\begin{equation}\label{3-1-4}\begin{array}{lll}
\D_{\mu}\big((f_1,0)\big)&=&\Big(({\rm id}, 0) \otimes (\ptl^{\circ}, 0) -
(\ptl^{\circ}, 0) \otimes
({\rm id}, 0)\Big) \mu^{\circ} (f_1, 0)\\
&=&\sum\limits_{f_1}\Big((f_1^{(1)}, 0) \otimes
\big(\ptl^{\circ}(f_1^{(2)}), 0\big) -
\big(\ptl^{\circ}(f_1^{(1)}), 0\big) \otimes (f_1^{(2)}, 0)\Big),\\
\D_{\mu}\big((0, f_2)\big)&=& ((\tau, id)\cdot(((0, id)\otimes (0, \ptl^{\circ}))-(id, \tau)\cdot (0, \ptl^{\circ})\otimes (0, id)))\mu^{\circ}((0, f_2))\\
&=& \sum\limits_{f_2}\big((f_2^{(1)}, 0) \otimes \big(0,
\ptl^{\circ}(f_2^{(2)}))-(0, \ptl^{\circ}(f_2^{(1)})) \otimes (f_2^{(2)}, 0)\big)\big),
 \end{array}\end{equation}
where the second equation is obtained by the following, for $(x^m, x^n), (x^k, x^l) \in \SS$, 
 $$\begin{array}{lll}\D_{\mu}(0, f_2) ((x^m, x^n) \otimes (x^k, x^l))= (0, f_2)([(x^m, x^n), (x^k, x^l)])\\
 \phantom{XX}
 =(0, f_2)(x^m\cdot \ptl(x^k)-\ptl(x^m)\cdot x^k, x^m \cdot \ptl(x^l) - \ptl(x^n) \cdot x^k)\\
 \phantom{XX}
 =f_2(x^m \cdot \ptl(x^l))-f_2(\ptl(x^n) \cdot x^k)\\
 \phantom{XX}
 =((f_2^{(1)}, 0)\otimes (0, \ptl^{\circ}(f_2^{(2)})) -(0, \ptl^{\circ}(f_2^{(1)})) \otimes (f_2^{(2)}, 0))((x^m, x^n) \otimes (x^k, x^l))
.  \end{array}$$
It is easy to verify that $(\LL^{\circ}_{\mu}, \D_{\mu})$ is a Lie
coalgebra.

We remark that 
since $(\AA, \mu_2)$ is an associative commutative algebra,
the coalgebra \linebreak $(\AA^{\circ}, \mu_2^{\circ})$ is a coassociative
cocommutative algebra. By the second equation of $(\ref{1-1-2})$, we
have $$\Big((0,{\rm id})\otimes (0,{\rm id} )-T\cdot \big((0, {\rm id}) \otimes (0, {\rm id})\big)\Big)\mu^{\circ}(0, f_2) =0\mbox{ for all $(0,f_2) \in
\LL^{\circ}_\mu$}.
%
%
$$
We also remark that 
the difference between  $\LL^{\circ}_{\mu}$ and
$\LL^{\circ}_{\v}$ is that $\LL^{\circ}_{\mu}$ is the maximal good
subspace of $\LL^{\ast}$ under the map $\mu^{\ast}$, where $\mu$ is the
multiplication of associative commutative algebra $(\LL, \mu)$; while
$\LL^{\circ}_{\v}$ is the maximal good
 subspace of $\LL^{\ast}$ under the map $\v^{\ast}$, where $\v$ is
 the Lie bracket of Lie algebra $(\LL, \v).$
\begin{prop} The Lie coalgebra $(\LL^{\circ}_{\mu}, \D_{\mu})$ is a
Lie subcoalgebra of $(\LL_{\v}^{\circ}, \D_{\v}).$
\end{prop}
\begin{proof}
For $f\in \LL^{\circ}_{\mu}=\AA^{\circ}_{\mu}
\oplus \AA^{\circ}_{\mu},\ f=(f_1, f_2),\ f_1,\ f_2 \in
\AA^{\circ}_{\mu},$ then
\begin{eqnarray}\begin{array}{lll}\D_{\v} (f) &=& \D_{\v}
\big((f_1, 0) + (0, f_2)\big)\\
& =& \Big(\big(\mu\big(({\rm id}, 0) \otimes (\ptl, 0)- (\ptl, 0) \otimes
({\rm id},0)\big)\big)^{\ast}\Big)\big((f_1,
0)\big) \\
&&+ (\mu \cdot((\tau, id)\cdot(((0, id)\otimes (0, \ptl)))-(id, \tau)\cdot( (0, \ptl)\otimes (0, id))))^{\ast}(0,
f_2)\\
&=& \sum\limits_{f_1} \Big((f_1^{(1)}, 0) \otimes \big(\ptl^*(f_1^{(2)}),
0\big)-\big(\ptl^*(f_1^{(1)}), 0\big) \otimes (f_1^{(2)}, 0)\Big)\\
&&+ \sum\limits_{f_2} \big((f_2^{(1)}, 0)\otimes (0,
\ptl^{\ast}(f_2^{(2)})-(0, \ptl^{\ast}(f_2^{(1)})) \otimes (f_2^{(2)}, 0)\big)\\&=&\D_{\mu}(f).\\[-29pt]
\end{array}
\end{eqnarray}\end{proof}

Let $\AA$ be an associative commutative
algebra over a field $\cal{K}$, ${\cal L} = \AA \oplus \AA$ be the direct
sum of $\AA$ with itself. Then $({\cal L}, \mu)$ is also an associative
commutative algebra with the multiplication $\mu\big((a, b), (c, d)\big) =
\big(\mu_1(a, c), \mu_2(b,d)\big)\in {\cal L}$ (where $\mu_1=\mu_2$ are
multiplications of the first and second copy $\AA$ of $\cal L$). Denote
${\rm Der}(\AA)$ the derivation vector space of $\AA$. For $\ptl \in
{\rm Der}(\AA)$, define the bracket $\v$ as following.
\begin{equation}\label{b-1-1}\begin{array}{l}\v\big((a,0), (b, 0)\big)= \Big(\mu\big(({\rm id}, 0) \otimes (\ptl,
0)-(\ptl, 0) \otimes ({\rm id},0)\big)\Big)\big((a,0), (b, 0)\big),\\[4pt]
 \v\big((0,c), (0,
d)\big)=\mu\Big((0, {\rm id})\otimes (0, {\rm id})-\big((0, {\rm id})\otimes (0, {\rm id})\big)\cdot
T\Big)\big((0,c), (0, d)\big)=0, \\[4pt]
\v\big((a,0), (0, c)\big) =\mu\Big(\big((0, {\rm id})\otimes(0, \ptl)\big)\cdot(\tau,
{\rm id})\Big)\big((a,0), (0, c)\big),\\
\v\big((0,c), (a, 0)\big) =-\v\big((a,0), (0, c)\big) .\end{array}\end{equation}
Then $({\cal L}, \v)$ is a Lie algebra. For convenience, we denote ${\cal L}
= \AA_1 \oplus \AA_2,\ \AA_1 = \AA_2 =\AA.$

\begin{theo}\label{Th-3-1}
Let $\cal{L}=(\AA_1 \otimes \AA_2, \mu)$ be the above commutative
associative algebra over any field ${\cal K}$ with characteristic
different from $2$, and $(\cal{L}, \v)$ be the Lie algebra defined by
\eqref{b-1-1}. If there exists $H=(h, h) \in {\cal L}$ such that the idea
of $({\cal L}, \mu)$ which is generated by
$\big(2\ptl(h), \ptl(h)\big)$ has a finite codimension, then
${\cal L}^{\circ}_{\mu} = {\cal L}^{\circ}_{\v}.$ In particular, $\LL^{\circ}_{\mu} = \LL^{\circ}_{\v}.$
\end{theo}
\begin{proof}
Since for $f\in {\cal L}^{\ast} = (\AA_1 \oplus
\AA_2)^{\ast}$ and $(a_1, a_2) \in \AA_1 \oplus \AA_2,$ we have
$f\big((a_1, a_2)\big)= f\big((a_1, 0)\big) + f\big((0, a_2)\big)$. Set $f\big((a_1, 0)\big) =
\big(f_1(a_1),0\big)=f_1(a_1)$, $f\big((0, a_2)\big)= \big(0,f_2(a_2)\big)=f_2(a_2)$.
It follows that $\ f_i \in \AA_i^{\ast},\ i=1,2 $, and we get ${\cal L}^{\ast}=
(\AA_1 \oplus \AA_2)^{\ast} = \AA_1^{\ast} \oplus \AA_2^{\ast}.$

Denote by $\cdot$ and $*$ the actions of $({\cal L}, \mu)$ and $({\cal L}, \v)$
on ${\cal L}^{\ast}$, respectively. Then for $f \in {\cal L}^{\ast},$ $w,v\in
{\cal L}$, we have $(f\cdot w)(v) = f\big(\mu(w, v)\big),\ (f*w)(v)= f\big(\v(w, v)\big).$ Let
$w=(w_1, w_2),$ $v=(v_1, v_2),\, H=(h, h) \in{\cal L}$, we obtain
\begin{eqnarray}\label{eq-3-2}
\v\big(w, \mu(v, H)\big) -\v\big(\mu(w, H), v\big)
&\!\!=\!\!&\v\big((w_1,w_2), (v_1h, v_2h)\big) -\v\big((w_1h, w_2h), (v_1, v_2)\big)\nonumber\\
&\!\!=\!\!&\big(w_1\ptl(v_1h) - \ptl(w_1)v_1h, w_1\ptl(v_2h) -
v_1h\ptl(w_2)\big)\nonumber\\
&&- \big(w_1h\ptl(v_1) - \ptl(w_1h) v_1, w_1h\ptl(v_2) -
v_1\ptl(w_2h)\big)\nonumber\\
&\!\!=\!\!&\big(2w_1v_1\ptl(h), w_1v_2\ptl(h) + v_1w_2\ptl(h)\big).
\end{eqnarray}
%
%
By  (\ref{eq-3-2}) and
\begin{equation}
 \begin{array}{lll}\big((f*w)\cdot H - f*\mu
(w, H)\big)(v_1, 0) &\!\!\!=\!\!\!& f\big(2w_1v_1\ptl(h),  v_1w_2\ptl(h)\big)\\
&\!\!\!=\!\!\!& \Big(f\cdot \big(2w_1\ptl(h), 0\big) + f\cdot\big(0, w_2 \ptl(h)\big)\cdot \tau\Big)
(v_1, 0),\\
\big((f*w)\cdot H - f*\mu (w, H))(0,v_2)\big) &\!\!\!=\!\!\!& f\big(0, w_1v_2\ptl(h)\big)\\ &\!\!\!=\!\!\!&
f\cdot\big(0, w_1\ptl(h)\big)(0, v_2),\nonumber
\end{array}
\end{equation}
%
we obtain
\begin{equation}\label{eq-3-3}\begin{array}{l}
\big((f*w)\cdot H - f*\mu (w, H)\big)_1=f_1\cdot\big(2w_1\ptl(h)\big)+ \Big(f_2\cdot\big(
w_2\ptl(h)\big)\Big)\cdot \tau,\\
\big((f*w)\cdot H - f*\mu (w, H)\big)_2=f_2\cdot\big(w_1\ptl(h)\big),
\end{array}
\end{equation}
where the subscript ``\,$_i$\,'' denotes the $i$-th coordinate of an element for $i=1,2$.
 If $f \in {\cal L}^{\circ}_{\v},$ i.e., $f *{\cal L}$ is finite dimensional,
then the left sides of (\ref{eq-3-3}) are in finite
dimensional subspaces. If the idea $\big(\AA \ptl(h)\big)$ of $\AA$ generalized by
$\ptl(h)$ has finite codimension in $\AA$, then the second equation of \eqref{eq-3-3} shows that $f_2 \cdot \AA$ is
finite dimensional, i.e., $f_2 \in \AA^{\circ}$. Moreover, if the
idea $\big(\AA\big(2\ptl(h)\big)\big)$ of $\AA$ which is generalized by $2\ptl(h)$ also has
finite codimension, then the first equation of \eqref{eq-3-3} shows that $f_1 \cdot \AA$ is finite
dimensional, i.e., $f_1 \in\AA^{\circ}$.
\end{proof}

By Theorem \ref{Th-3-1}, we have $\LL^{\circ}_{\mu} =
\LL^{\circ}_{\v}$, which is now denoted by $\LL^{\circ}$. By Lemmas \ref{lem-3-2} and \ref{lem-3-1}, we have
\begin{eqnarray}\label{MaxHVir}
\!\!\!\!\!\!\!\!\!\!\!\!&\!\!\!\!\!\!\!\!\!\!\!\!&
\LL^{\circ}={\cal A}^{\circ}\oplus{\cal A}^{\circ} \mbox{ with }{\cal A}=\F[x^{\pm1}]\mbox{ and }\\
\!\!\!\!\!\!\!\!\!\!\!\!&\!\!\!\!\!\!\!\!\!\!\!\!&
{\cal A}^{\circ}=\Big\{f=\mbox{$\sum\limits_{i\in\Z}$}f_i\vp^i\in{\cal A}^*\,\Big|\,\exists\, r\in\N,c_j\in\F\mbox{ such that }
f_n=\mbox{$\sum\limits_{j=1}^r$}c_jf_{n-j},\,\forall\,n\in\Z\Big\}.\nonumber
\end{eqnarray}


\section{Dual \ Lie \ bialgebras \ of \ the \ twisted \ Heisenberg-Virasoro \ types
}\setcounter{section}{3}\setcounter{theo}{0}\setcounter{equation}{0}

As stated in the introduction, the Heisenberg-Virasoro algebra was first studied in
\cite{A}, it is  an important algebra structure which has close
relations with the Heisenberg algebra and the Virasoro algebra, and has also
some relations with the full-toroidal Lie algebras and conformal
algebras (see, e.g., \cite{S}). The representations of the twisted
Heisenberg-Virasoro algebra were studied by some authors (see
\cite{B, L-C, L-W, L-Z}). The authors of \cite{L-P} investigated the
Lie bialgebra structures of the twisted Heisenberg-Virasoro algebra. In
 this section, we will investigate the dual Lie bialgebra
structures of the twisted Heisenberg-Virasoro type. First, we recall some
results which are related to the Heisenberg-Virasoro Lie bialgebra
(see, e.g., \cite{9,L-P}).

\begin{prop}\label{p4-1}\begin{enumerate}
\item[\rm(1)] Let $\cal{W}$ be the classical Witt $($or Virasoro$)$ algebra $($i.e., $\cal{W}\!=\!\F[t,t^{-1}]$ such that
 $[f,g]=f\frac{dg}{dt}-g\frac{df}{dt}$ for $f,g\in\cal{W}$$)$.
\begin{itemize}\item[\rm(i)]Every Lie bialgebra structure on $\cal{W}$
 is  coboundary triangular associated to a solution $r$ of CYBE \eqref{CYBE} of the form
 $r= a\otimes b- b\otimes a$ for some nonzero $a, b \in\cal{W}$ satisfying \begin{equation}\label{2dim}
 [a,b]=k{\ssc\,}b\mbox{ \ for some \ }0\ne k\in\F.\end{equation}
\item[\rm(ii)]
Let  $\gg $ be an infinite dimensional Lie subalgebra of $\cal{W}$
such that $t \in \gg $ and $\gg  \ncong \cal{W}$ as Lie algebras.
Denote by $\gg ^{(n)}$ the Lie bialgebra defined on $\gg $
associated to the solution $r_n = t \otimes t^n - t^n \otimes t$ of
 CYBE for any $t^n \in \gg $. Then every Lie bialgebra structure
on $\gg $ is isomorphic to $\gg ^{(n)}$  for some $n$ with $t^n \in
\gg $.
\end{itemize}
\item[\rm(2)] Let $(\LL, \v)$  be the centerless twisted Heisenberg-Virasoro algebra defined by $(\ref{1-1-2}),$ and
 $(\LL, \v, \D)$ be a Lie biagebra. Then $\D = \D_r+\d$ such that $\D_r$ is defined by
$\D_r (x) = x\cdot r \in \LL \otimes \LL$ for some $r \in {\rm Im}(1-\tau)$, where $\tau$ is defined by
$\tau(a\otimes b) = b\otimes a$, and $\d$ is defined by \begin{equation}\label{d-defby}\d(L_n) = (n\a + \g)(I_0 \otimes
I_n - I_n \otimes I_0),\ \d(I_n) = \b(I_0 \otimes I_n - I_n \otimes
I_0),\end{equation} for some fixed $\a,\,\b,\,\g \in\C.$ Furthermore, $(\LL, \v, \d)$
is a Lie
bialgebra. 
%
%
%
%
%
\end{enumerate}
\end{prop}
\begin{rema}\rm
It is proved in \cite{4} that if two elements $a,b$ in a Lie
algebra $({\cal L}, [\cdot, \cdot])$ satisfy \eqref{2dim},
then $r= a \otimes b - b\otimes a$ is a solution of CYBE, and one obtains
a coboundary triangular Lie bialgebra $({\cal L} , [\cdot, \cdot],
\D_r)$ by defining $\D_r: {\cal L} \rar {\cal L} \otimes {\cal L}$
as $\D_r (z) = z \cdot r$ for $z\in {\cal L}$.
Proposition \ref{p4-2} below shows that for $a=(x, 0), b=(x, qx^m) \in \LL$ with $0\ne q\in \F,\,m\in \Z$,
even though \eqref{2dim} does not hold, we still have a  solution of CYBE $r=a \otimes b - b\otimes a$. Thus, \eqref{2dim} is not the necessary condition for
$r=a \otimes b - b\otimes a$ to be a  solution of CYBE.
\end{rema}
\begin{prop}\label{p4-2} Let $(\LL, \v)$ be the centerless twisted
Heisenberg-Virasoro algebra defined by $(\ref{1-1-2})$. Then $r= (x, 0) \otimes (x^{m+1}, qx^n) - (x^{m+1}, qx^n)
\otimes (x, 0)$ with $m,n\in\Z,\,q
\in \F$ is a solution of CYBE if
and only if one of the following holds$:$ $(1)\ m=n;\ (2)\ m=0;\ (3)\ q=0$. Furthermore, $r= (x, 0) \otimes (0, qx^n) - (0, qx^n)\otimes(x, 0)$
is a solution of CYBE.
%
%
%
\end{prop}
\begin{proof}
By computation, we can get \begin{eqnarray}
&&[r_{12}, r_{13}] + [r_{12}, r_{23}] + [r_{13},
r_{23}]\nonumber\\
&&\phantom{[r_{12},}=(n-m)(x^{m+1}, 0) \otimes (0, qx^n) \otimes (x, 0)+ (m-n)(x^{m+1}, 0)\otimes (x, 0)\otimes (0, qx^n)\nonumber\\
&&\phantom{[r_{12},=}+(m-n)(0, qx^n) \otimes (x^{m+1}, 0) \otimes (x, 0)-(m-n)(0, qx^n)\otimes (x, 0)\otimes (x^{m+1}, 0)\nonumber\\
&&\phantom{[r_{12},=}-(m-n)(x, 0) \otimes (x^{m+1}, 0)\otimes (0, qx^n)+(m-n)(x, 0)\otimes (0, qx^n)\otimes (x^{m+1}, 0).\nonumber
\end{eqnarray}


\noindent Then $[r_{12}, r_{13}] + [r_{12}, r_{23}] + [r_{13}, r_{23}]=0$ if
and only if one of the four cases in Proposition \ref{p4-2} occurs.
\end{proof}

The authors of \cite{16, SS}  constructed the dual Lie bialgebra
structures of Witt (Virasoro) and Poisson type Lie bialgebras.
The dual structures of loop type and current type Lie bialgebras were considered in \cite{SS-2}.
Now, we start to investigate the dual Lie
bialgebra structures of the twisted Heisenberg-Virasoro type.
Let $\LL\! =\! \AA_1 \oplus \AA_2$, $\AA_1\! =\! \AA_2\! =\!\C[x^{\pm 1}]$ and $\SS =
\{ (x^{m}, x^n)\, |\, m,\, n \in \Z\}$ be the standard basis of $\LL$.
Denote by $\SS^* = \{(\vp^i, \vp^j) \,|\, i,\, j \in \Z\}$ the set of
dual basis of $\SS$, i.e., $\langle(\vp^i, 0), (x^j, 0)\rangle
=\langle(0, \vp^i), (0, x^j)\rangle = \d_{i,j}$ for $i,j\in \Z,$
and $\langle(\vp^i, 0), (0 ,
x^j)\rangle \!=\! 0$, $\langle(0, \vp^i), (x^j , 0)\rangle = 0.$

\begin{theo}\label{Th-4-1}Let $(\LL, \v)$ be the twisted Heisenberg-Virasoro
algebra defined by $(\ref{1-1-2})$, and $(\LL^{\circ}, \D)$ be a Lie coalgebra,
 where $\LL^{\circ}$ is determined by Lemmas
$\ref{lem-3-2},\ \ref{lem-3-1}$ and Theorem $\ref{Th-3-1}$, then the
cobracket $\D$ is uniquely determined, for $m\in\Z$, by
\begin{enumerate}
\item[\rm(1)] $\D\big((\vp^m, 0)\big) =\sum_{i,j\in\Z,\,i+j=m+1}(j-i) (\vp^i, 0)\otimes
(\vp^j, 0),
$

\item[\rm(2)] $\D\big((0, \vp^m)\big) =\sum_{i,j\in\Z,\,i+j=m+1} (j(\vp^i,
0)\otimes (0, \vp^j)- i(0, \vp^i) \otimes (\vp^j, 0))
.$
\end{enumerate}
\end{theo}
\begin{proof}
For $\vp^m \in \AA^{\ast}$, assume
$\mu_1^{\ast} (\vp^m) = \mu_2^{\ast}(\vp^m) =
\sum_{i,j}c_{i,j} \vp^i \otimes \vp^j$ for some $ c_{i,j}
\in \C,$ then $c_{i,j}=\sum_{i,j}c_{i,j} \vp^i \otimes \vp^j
(x^i \otimes x^j) = \mu_k^{\ast}(\vp^m)(x^i\otimes
x^j)=\vp^m(x^{i+j}) = \d_{m, i+j}$. Therefore $\mu_k^{\ast}(\vp^m) =
\sum_{i+j=m} \vp^i \otimes \vp^j,$ $i,j\in \Z$ and $k=1, 2.$ Assume
 $\ptl^{\ast}(\vp^m) = \sum_{i\in \Z} c_i \vp^i$ for some $c_i\in\C$. Then
$c_i= \sum_{i\in \Z} c_i \vp^i(x^i) =
\ptl^{\ast}(\vp^m)(x^i)= \vp^m\big(\ptl(x^i)\big) = i\d_{m, i-1}.$ From
this, we have $\ptl^{\ast}(\vp^m) = (m+1)\vp^{m+1}$. By
$(\ref{eq-3-1})$ and $(\ref{3-1-4})$, for $i,j \in \Z$, we obtain
\\[4pt]
$\ \hspace*{40pt}
\D\big((\vp^m, 0)\big) = \sum\limits_{i+j=m} (\vp^i, 0) \otimes
\big(\ptl^{\ast}(\vp^j), 0\big)- \big(\ptl^{\ast}(\vp^i), 0\big) \otimes (\vp^j,
0)$\\[4pt]
$\ \hspace*{40pt}
\phantom{\D\big((\vp^m, 0)\big)}
=\sum\limits_{i+j=m+1}(j-i) \big((\vp^i, 0)\otimes
(\vp^j,0)\big),
$
\\[7pt]
$\ \hspace*{40pt}
\D\big((0, \vp^m)\big) = \sum\limits_{i+j=m}\big((\vp^i, 0)\otimes \big(0,
\ptl^{\ast}(\vp^j)\big)- (0, \ptl^{\ast}(\vp^i)) \otimes (\vp^j, 0)\big)
$
\\[4pt]
$\ \hspace*{40pt}
\phantom{\D\big((0, \vp^m)\big)}
=\sum\limits_{i+j=m}\big((j+1) (\vp^i,
0)\otimes (0, \vp^{j+1})-(i+1)(0, \vp^{i+1}) \otimes (\vp^j, 0)\big)
$
\\[4pt]
$\ \hspace*{40pt}
\phantom{\D\big((0, \vp^m)\big)}
=\sum\limits_{i+j=m+1} (j (\vp^i,
0)\otimes (0, \vp^j)-i(0, \vp^i) \otimes (\vp^j, 0) ).
$
\end{proof}

\begin{theo}\label{Th-4-2} Let $(\LL, \v, \D_r)$ be a coboundary triangular Lie
bialgebra related to the solution of CYBE $r = (x,
0) \otimes (x^{m+1}, 0) - (x^{m+1}, 0) \otimes (x, 0)$ with $m\ne0$. Then the
dual Lie bialgebra of $(\LL, \v, \D_r)$ is $(\LL^{\circ}, [\cdot,
\cdot], \D)$, where the underline vector space $\LL^{\circ}$ is
determined by \eqref{MaxHVir}
, the cobracket $\D$ is determined by Theorem
$\ref{Th-4-1}$ and the bracket is uniquely determined by
\begin{itemize}\item[{\rm (1)}]
$[(\vp^i, 0), (\vp^j, 0)] =
\left\{\begin{array}{llll}(2m-j+1)(\vp^{j-m},0)  &\mbox{if}\ \ i=1,\, j\neq 1,\\
(j-1)(\vp^{j}, 0) &\mbox{if}\ \ i= m+1,\, j\neq 1, m+1,\\
0 &\mbox{if}\ \ i,\, j
\not\in \{1, m+1\}.
\end{array}\right.
$
\item[{\rm (2)}] $[(\vp^i, 0), (0, \vp^j)]=\left\{\begin{array}{llll}(m-j)(0, \vp^{j-m}) &\mbox{if } \ i=1,\\
 j(0, \vp^j) &\mbox{if }\ i=m+1,\\
0 &\mbox{if}\ \ i \not\in\{1,
m+1\}.\end{array}\right. $
\item[\rm(3)] $[(0, \vp^i), (0, \vp^j)] = 0$ for $i, j \in \Z.$
\end{itemize}
\end{theo}

\begin{conv}\rm\begin{enumerate} \item[\rm(1)] In the dual Lie bialgebra $(\LL^{\circ}, [\cdot, \cdot],
\D)$, we always use $[\cdot, \cdot]$ to denote its Lie bracket and
 $\D$ to denote its Lie cobracket
, i.e., $[\cdot, \cdot]
= \D_r^{\circ},\ \D= \v^{\circ}.$
\item[\rm(2)] For $f \in \LL^{\circ} =\AA^{\circ} \oplus \AA^{\circ},\ f=(f_1,
f_2)$ and $(x^k, x^l) \in \AA \oplus \AA,$ we always write $f(x^k, x^l)=(f_1,
f_2)\big((x^k, x^l)\big) = f_1(x^k) + f_2(x^l)=
\big(f_1(x^k), f_2(x^l)\big).$
\end{enumerate}
\end{conv}
\ni{\it Proof of Theorem $\ref{Th-4-2}$.~}~ By computation, we can get
\begin{eqnarray*}&&\Big\langle \big[(\vp^i, 0), (\vp^j, 0)\big], (x^k,
x^l)\Big\rangle\nonumber\\
&&\phantom{\langle [(\vp^i}=\Big\langle(\vp^i, 0)\otimes (\vp^j, 0), \D_r\big((x^k,
x^l)\big)\Big\rangle\nonumber\\
&&\phantom{\langle [(\vp^i}=\Big\langle(\vp^i, 0)\otimes (\vp^j, 0),  \big((1-k)x^k, -lx^l)\big)
\otimes (x^{m+1}, 0)
- (x^{m+1}, 0)\otimes \big((1-k)x^k, -lx^l\big)
\\
&&\phantom{\langle [(\vp^i=}+ (x, 0) \otimes \big((m+1-k)x^{m+k}, -lx^{m+l}\big) 
- \big((m+1-k)x^{m+k}, -lx^{m+l}\big) \otimes (x, 0)\Big\rangle \nonumber\\
&&\phantom{\langle [(\vp^i}= \big((1-k)\d_{i,k}, 0\big)(\d_{j, m+1}, 0) + (\d_{i,1},
0)\big((m+1-k)\d_{m+k, j}, 0\big)\nonumber\\
&&\phantom{\langle [(\vp^i=} -(\d_{i, m+1}, 0)\big((1-k)\d_{j,k}, 0\big)-(\d_{j,1},
0)\big((m+1-k)\d_{m+k,i}, 0\big)\nonumber\\
&&\phantom{\langle [(\vp^i}=\big((1-i)\d_{i,k}, 0\big)(\d_{j, m+1}, 0)+ (\d_{i,1},
0)\big((2m+1-j)\d_{j-m, k}, 0\big)\nonumber\\
&&\phantom{\langle [(\vp^i=} -(\d_{i, m+1}, 0)\big((1-j)\d_{j,k}, 0\big)-(\d_{j,1},
0)\big((2m+1-i)\d_{i-m,k}, 0\big)\nonumber\\
&&\phantom{\langle [(\vp^i}=\Big\langle\d_{j, m+1}\big((1-i)\vp^i, 0\big)+\d_{i,1}\big((2m+1-j)\vp^{j-m}, 0\big)\nonumber\\
&&\phantom{\langle [(\vp^i=} -\d_{i, m+1}\big((1-j)\vp^j, 0\big)-\d_{j,1}\big((2m+1-i)\vp^{i-m}, 0\big),(x^k,
x^l)\Big\rangle.\nonumber
\end{eqnarray*}
\noindent We obtain
\begin{equation}\label{eq-4-1}\begin{array}{lll}
[(\vp^i, 0), (\vp^j, 0)]&=&\d_{j, m+1}\big((1-i)\vp^i, 0\big)+ \d_{i,1}\big((2m+1-j)\vp^{j-m}, 0\big)\\
&& -\d_{i, m+1}\big((1-j)\vp^j, 0\big)-\d_{j,1}\big((2m+1-i)\vp^{i-m}, 0\big).
\end{array}
\end{equation}
From this, we obtain $[(\vp^i, 0), (\vp^j, 0)]=0$
if $i, j \not\in \{1, m+1\}.$ If $i=1\neq j,$ by noting that
$m\neq 0,$ we have $[(\vp^1, 0), (\vp^j,
0)]=\big((2m+1-j)\vp^{j-m}, 0\big)$ by (\ref{eq-4-1}) for $j\in \Z.$ If $i=m+1$ and $j\neq 1,\, m+1,$
then $[(\vp^{m+1}, 0), (\vp^j, 0)]=\big((j-1)\vp^j, 0\big).$ The case (1)
of Theorem \ref{Th-4-2} holds.

For the case (2), since
\begin{equation}
\begin{array}{lll}\Big\langle [(\vp^i, 0), (0, \vp^j)], (x^k,
x^l)\Big\rangle &=&\Big\langle(\vp^i, 0)\otimes(0, \vp^j), \D_r\big((x^k,
x^l)\big)\Big\rangle \\
&=&( \d_{i,1},0)(0, -l\d_{j,m+l})- (\d_{i,m+1}, 0)(0, -l\d_{j,l})\\
&=&\Big\langle \d_{i,1}\big(0, (m-j)\vp^{j-m}\big)- \d_{i,m+1}(0, -j\vp^j),
(x^k, x^l)\Big\rangle,
\end{array}\nonumber
\end{equation}

\noindent we have
\begin{equation}\label{eq-4-2}
[(\vp^i, 0), (0, \vp^j)]=\d_{i,1}\big(0, (m-j)\vp^{j-m}\big)- \d_{i,m+1}(0,
-j\vp^j).
\end{equation}
From this, it follows that
$[(\vp^i, 0), (0, \vp^j)]=0$ if $i \not\in \{1, m+1\}$. If $i=1$, 
we have $[(\vp, 0), (0, \vp^j)]$ $=\big(0, (m-j)\vp^{j-m}\big)$; if
$i=m+1$, we have $[(\vp^{m+1}, 0), (0, \vp^j)] = (0, j\vp^j).$
The case (2) of the theorem is proved.

Finally, for the case (3), since
$\big\langle [(0, \vp^i), (0, \vp^j)], (x^k, x^l)\big\rangle = \big\langle (0,
\vp^i)\otimes (0, \vp^j),\D_r\big((x^k, x^l)\big)\big\rangle$ $=0$ for $i,j,k,
l\in \Z,$ we have $[(0, \vp^i), (0, \vp^j)]=0.$
\QED\vskip4pt

\begin{theo}\label{Th-4-3} Let $(\LL, \v, \D_r)$ be the coboundary triangular Lie
bialgebra which is related to the solution of CYBE $r=(x, 0) \otimes (x^{m+1}, qx^m) - (x^{m+1}, qx^m) \otimes
(x,0),$ then its dual Lie bialgebra is $(\LL^{\circ}, [\cdot,
\cdot], \D)$, where $\LL^{\circ}$ is determined by \eqref{MaxHVir}
, the
cobracket $\D$ is determined by Theorem $\ref{Th-4-1}$, and the
bracket is uniquely determined by
\begin{enumerate}
\item[{\rm (1)}] $[(\vp^i, 0), (\vp^j, 0)] = \left\{\begin{array}{llll}
(2m-j+1)(\vp^{j-m},0) &\mbox{if}\ i=1,\, j\neq 1,\\
(j-1)(\vp^{j}, 0) &\mbox{if}\ i= m+1,\, j\neq 1,\, m+1,\\
0 &\mbox{if}\ i,\, j
\not\in \{1,\, m+1\}.\end{array}\right.$
\item[{\rm (2)}] $[(\vp^i, 0), (0, \vp^j)]=\left\{\begin{array}{llll}mq(\vp^{j-m+1}, 0) -(j-m)(0, \vp^{j-m}) &\mbox{if } \ i=1,\, m\neq 0,\\
-mq(\vp^{m+1}, 0) +m(0, \vp^m) &\mbox{if}\  i=m+1,\, j=m,\\
 j(0, \vp^j) &\mbox{if }\ i=m+1,\, j\neq m,\\
(1-i)q(\vp^i, 0 ) &\mbox{if}\ i \neq 1,\, m+1,\, j=m,\\
0 &\mbox{otherwise}.\end{array}\right. $
\item[{\rm (3)}] $[(0, \vp^i), (0, \vp^j)] =\left\{\begin{array}{llll} jq(0, \vp^j) &\mbox{if}\ i=m,\, j\neq
m,\\
0 &\mbox{if}\ i\neq m,\, j\neq m.
 \end{array}\right.$
\end{enumerate}
\end{theo}
\begin{proof}
We can compute that
\begin{eqnarray}
&&\Big\langle [(\vp^i, 0), (\vp^j, 0)], (x^k,
x^l)\Big\rangle\nonumber\\
&&\phantom{\langle [(\vp^i}=\Big\langle (\vp^i, 0)\otimes (\vp^j, 0), \D_r\big((x^k,
x^l)\big)\Big\rangle\nonumber\\
&&\phantom{\langle [(\vp^i}=\Big\langle (\vp^i, 0)\otimes (\vp^j, 0),\ (x^k,
x^l)\cdot\big((x,0)\otimes (x^{m+1}, qx^m)-(x^{m+1}, qx^m)\otimes
(x, 0)\big)\Big\rangle\nonumber\\
&&\phantom{\langle [(\vp^i}=\Big\langle (\vp^i, 0)\otimes (\vp^j, 0),\ \big((1-k)x^k, -lx^l\big)\otimes (
x^{m+1}, qx^m)\nonumber\\
&&\phantom{\langle [(\vp^i=}-( x^{m+1}, qx^m)\otimes \big((1-k)x^k, -lx^l\big)\nonumber \\
&&\phantom{\langle [(\vp^i=}+ (x, 0) \otimes \big((m+1-k)x^{m+k}, mqx^{m+k-1}-lx^{m+l}\big)\nonumber\\
&&\phantom{\langle [(\vp^i=}- \big((m+1-k)x^{m+k}, mqx^{m+k-1}-lx^{m+l}\big)\otimes (x, 0)\Big\rangle\nonumber\\
&&\phantom{\langle [(\vp^i}=\big((1-k)\d_{i,k}, 0\big)\d_{j, m+1} + \d_{i,1}\big((m+1-k)\d_{j,m+k},0\big)-\big((1-k)\d_{j,k}, 0\big)\d_{i, m+1} \nonumber\\
&&\phantom{\langle [(\vp^i=}- \d_{j,1}\big((m+1-k)\d_{i,m+k},0\big)\nonumber\\
&&\phantom{\langle [(\vp^i}=\Big\langle \d_{j, m+1}\big((1-i)\vp^i, 0\big)+
\d_{i,1}\big((2m+1-j)\vp^{j-m},0\big)-\d_{i, m+1}\big((1-j)\vp^j, 0\big)\nonumber\\
&&\phantom{\langle [(\vp^i=}- \d_{j,1}\big((2m+1-i)\vp^{i-m},0\big),\ (x^k,
x^l)\Big\rangle.\nonumber
\end{eqnarray}
Hence we have
\begin{equation}
\begin{array}{lll}[(\vp^i, 0), (\vp^j, 0)]&=&\d_{j,
m+1}\big((1-i)\vp^i, 0\big)+ \d_{i,1}\big((2m+1-j)\vp^{j-m},0\big)\\
&&-\d_{i, m+1}\big((1-j)\vp^j, 0\big)-
\d_{j,1}\big((2m+1-i)\vp^{i-m},0\big).\end{array}\nonumber
\end{equation}
Just as the proof of Theorem \ref{Th-4-2}, we obtain case (1)
of Theorem \ref{Th-4-3}.

For  case (2), by computation, we have
\begin{eqnarray}
&&\Big\langle [(\vp^i, 0), (0, \vp^j)], (x^k,
x^l)\Big\rangle\nonumber\\
&&\phantom{\langle [(\vp^i}=\Big\langle (\vp^i, 0)\otimes (0,\vp^j), \D_r\big((x^k,
x^l)\big)\Big\rangle\nonumber\\
&&\phantom{\langle [(\vp^i}=\Big\langle (\vp^i, 0)\otimes (0, \vp^j),\ (x^k,
x^l)\cdot\big((x,0)\otimes (x^{m+1}, qx^m)-(x^{m+1}, qx^m)\otimes
(x, 0)\big)\Big\rangle\nonumber\\
&&\phantom{\langle [(\vp^i}=\Big\langle (\vp^i, 0)\otimes (0, \vp^j), \big((1-k)x^k, -lx^l\big)\otimes (
x^{m+1}, qx^m)
\nonumber
\\
&&\phantom{\langle [(\vp^i=}-( x^{m+1}, qx^m)\otimes \big((1-k)x^k, -lx^l\big)\nonumber\\
&&\phantom{\langle [(\vp^i=}+ (x, 0) \otimes \big((m+1-k)x^{m+k}, mqx^{m+k-1}-lx^{m+l}\big)\nonumber\\
&&\phantom{\langle [(\vp^i=}- \big((m+1-k)x^{m+k}, mqx^{m+k-1}-lx^{m+l}\big)\otimes (x, 0)\Big\rangle
\nonumber\\
&&\phantom{\langle [(\vp^i}= q\d_{j,m}\big((1-k)\d_{i,k}, 0\big) + \d_{i,1}(0, mq\d_{j, m+k-1}
-l\d_{j,m+l})- \d_{i,m+1}(0, -l\d_{j,l})\nonumber\\
&&\phantom{\langle [(\vp^i}= \Big\langle \d_{j,m}\big((1-i)q\vp^i, 0\big) + \d_{i,1}\Big((mq\vp^{j-m+1}, 0\big)-\big(0, (j-m)\vp^{j-m}\big)\Big)\nonumber\\
&&\phantom{\langle [(\vp^i=} +\d_{i,m+1}(0, j\vp^j), \ (x^k, x^l)\Big\rangle.\nonumber
\end{eqnarray}
Thus\begin{equation}\label{eq-4-3}\begin{array}{lll}
[(\vp^i, 0), (0, \vp^j)]&=&\d_{j,m}\big((1-i)q\vp^i, 0\big) + \d_{i,1}\Big((mq\vp^{j-m+1}, 0\big)
-\big(0, (j-m)\vp^{j-m}\big)\Big)
\\&&  +\d_{i,m+1}(0, j\vp^j).
\end{array}\end{equation}
From this, if $i\neq 1,\, m+1$ and $j\neq m$,  we obtain
$[(\vp^i, 0), (0, \vp^j)]=0$. If $i=1$, then
$$[(\vp^1, 0), (0, \vp^j)]=(mq\vp^{j-m+1}, 0)-\big(0,
(j-m)\vp^{j-m}\big) +\d_{1,m+1}(0, j\vp^j).$$
Thus $[(\vp^1, 0), (0,
\vp^j)]=0$ if $m=0$, and $[(\vp^1, 0), (0,
\vp^j)]=(mq\vp^{j-m+1}, 0)-(0, (j-m)\vp^{j-m})$ if  $m\neq 0$. Assume
$i=m+1\neq 1$. Then (\ref{eq-4-3}) gives
\begin{equation}
[(\vp^{m+1}, 0), (0, \vp^j)]=\d_{j,m}(-mq\vp^{m+1}, 0) +(0,
j\vp^j)=\left\{\begin{array}{llll}(-mq\vp^{m+1}, 0) +(0, m\vp^m)
&\mbox{if} \ j=m,\\
(0, j\vp^j) &\mbox{if
}j\neq m.\end{array}\right.\nonumber
\end{equation}
Assume  $i \neq 1, m+1$. Then
\begin{equation}
[(\vp^i, 0), (0, \vp^j)]=\d_{j,m}\big((1-i)q\vp^i,
0\big)=\left\{\begin{array}{llll}\big((1-i)q\vp^i, 0\big) &\mbox{if}\ j=m,\\
0 &\mbox{if}\ j\neq
m.\end{array}\right.\nonumber
\end{equation}
Therefore case (2) is obtained.

 For case (3) of Theorem \ref{Th-4-3},
since
\begin{eqnarray}&&\Big\langle [(0, \vp^i), (0, \vp^j)], (x^k,
x^l)\Big\rangle\nonumber\\
&&\phantom{\langle [(0,}=\Big\langle (0, \vp^i)\otimes (0,\vp^j), \D_r\big((x^k,
x^l)\big)\Big\rangle\nonumber\\
&&\phantom{\langle [(0,}=\Big\langle (0, \vp^i)\otimes (0, \vp^j), (x^k,
x^l)\cdot\big((x,0)\otimes (x^{m+1}, qx^m)-(x^{m+1}, qx^m)\otimes
(x, 0)\big)\Big\rangle
\nonumber
\end{eqnarray}\begin{eqnarray}
&&\phantom{\langle [(0,}=\Big\langle (0, \vp^i)\otimes (0, \vp^j), \big((1-k)x^k, -lx^l\big)\otimes (
x^{m+1}, qx^m)\nonumber\\
&&\phantom{\langle [(0,=}-( x^{m+1}, qx^m)\otimes \big((1-k)x^k, -lx^l\big)\nonumber \\
&&\phantom{\langle [(0,=}+ (x, 0) \otimes \big((m+1-k)x^{m+k}, mqx^{m+k-1}-lx^{m+l}\big)\nonumber\\
&&\phantom{\langle [(0,=}- \big((m+1-k)x^{m+k}, mqx^{m+k-1}-lx^{m+l}\big)\otimes (x, 0)\Big\rangle\nonumber\\
&&\phantom{\langle [(0,}= (0, -l\d_{i,l})q\d_{j,m} -q\d_{i,m}(0, -l\d_{j,l})\nonumber\\
&&\phantom{\langle [(0,}=\Big\langle  \d_{i,m}(0, jq\vp^j) - \d_{j, m}(0, iq\vp^i), (x^k,
x^l)\Big\rangle,\nonumber
\end{eqnarray}

\noindent we have
\begin{equation}\label{eq-4-4}
[(0, \vp^i), (0, \vp^j)]=\d_{i,m}(0, jq\vp^j) - \d_{j, m}(0,
iq\vp^i)=\left\{\begin{array}{llll} (0, jq\vp^j) &\mbox{if}\ i=m, j\neq
m,\\
0 &\mbox{if}\ i\neq m, j\neq m.
 \end{array}\right.
\end{equation}
The theorem is proved.
\end{proof}

%

\begin{theo}\label{Th-4-4}\begin{enumerate}\item[{\rm (1)}] Let $(\LL, \v,
\D_r)$ be the coboundary triangular Lie bialgebra related to the
solution of CYBE  $r=(x, 0) \otimes (x, qx^m) - (x,
qx^m) \otimes (x, 0)$, then its dual Lie bialgebra is
$(\LL^{\circ}, [\cdot, \cdot], \D)$, where $\LL^{\circ}$ is defined
by \eqref{MaxHVir}
,
the cobracket $\D$ is determined by Theorem $\ref{Th-4-1}$, and the
bracket is uniquely determined by
\begin{enumerate}\item[{\rm (a)}] $[(\vp^i, 0), (\vp^j, 0)]=0;$
\item[{\rm (b)}] $([\vp^i,0), (0, \vp^j)] =\left\{\begin{array}{llll} mq(\vp^{j-m+1},0)
&\mbox{if}\ i=1,\\
(1-i)q(\vp^{i}, 0) &\mbox{if}\ i\neq 1,\, j=m,\\
0 &\mbox{if } i\neq1,\, j\neq m;
 \end{array}\right.$
\item[{\rm (c)}] $[(0, \vp^i), (0, \vp^j)]=\left\{\begin{array}{llll}jq (0, \vp^j) &\mbox{if}\ i=m,\, j\neq m,\\
 0 
&\mbox{if}\ i\neq m,\, j\neq m.\end{array}\right.$
\end{enumerate}
\item[{\rm (2)}] Let $(\LL, \v,
\D_r^{\prime})$ be the coboundary triangular Lie bialgebra
related to the solution of CYBE $r^{\prime}=(x, 0) \otimes
(0, qx^m) - (0, qx^m) \otimes (x, 0)$, then its dual Lie bialgebra
is $(\LL^{\circ}, [\cdot, \cdot], \D)$, where $\LL^{\circ}$ is
defined by \eqref{MaxHVir}
, the cobracket $\D$ is determined by Theorem
$\ref{Th-4-1}$, and the bracket is uniquely determined by

\begin{enumerate}
\item[{\rm (a)}] $[(\vp^i, 0), (\vp^j, 0)]=0;$

\item[{\rm (b)}] $[(\vp^i, 0), (0, \vp^j)]= \left\{\begin{array}{llll}mq (\vp^{j-m+1}, 0) &\mbox{if}\ i=1, \\
(1-i)q(\vp^i, 0)\  &\mbox{if}\ i\neq 1,\, j=m,\\
0 
 &\mbox{if  $i \neq 1,\, j\neq m $};\end{array}\right.$
\item[{\rm (c)}] $[(0, \vp^i), (0, \vp^j)]= \left\{\begin{array}{llll}jq(0, \vp^j ) &\mbox{if}\ i=m,\, j\neq
m, \\
0 &\mbox{if}\ i\neq m,\, j\neq m.
\end{array}\right. $
\end{enumerate}
\end{enumerate}
\end{theo}
\begin{proof}
For any $k,\ l \in \Z$,
\begin{eqnarray}&&\Big\langle [(\vp^i, 0), (\vp^j, 0)], (x^k,
x^l)\Big\rangle\nonumber\\
&&\phantom{\langle [(\vp^i,}= \Big\langle (\vp^i, 0) \otimes (\vp^j, 0), \D_r\big((x^k,
x^l)\big)\Big\rangle\nonumber\\
&&\phantom{\langle [(\vp^i,}= \Big\langle (\vp^i, 0) \otimes (\vp^j, 0), \big((1-k)x^k, -lx^l\big) \otimes
(x, qx^m)- (x, qx^m)\otimes \big((1-k)x^k, -lx^l\big)\nonumber\\
&&\phantom{\langle [(\vp^i,=}+ (x, 0) \otimes \big((1-k)x^k, mqx^{m+k-1}- lx^l\big)-\big((1-k)x^k, mqx^{m+k-1}- lx^l\big)\otimes (x, 0) \Big\rangle
\nonumber
\\
&&\phantom{\langle [(\vp^i,}=\d_{j,1}\big((1-k)\d_{i,k}, 0\big)+\d_{i, 1}\big((1-k)\d_{j,k}, 0\big)-\d_{i, 1}\big((1-k)\d_{j, k}, 0\big) - \d_{j, 1}\big((1-k)\d_{i, k}, 0\big) \nonumber\\
&&\phantom{\langle [(\vp^i,}=0,\nonumber
\end{eqnarray}
\noindent it gives that $[(\vp^i, 0), (\vp^j, 0)]=0$, so the first case of
(1) is obtained.

Now we prove (b) of case (1). Since
\begin{eqnarray}&& \Big\langle [(\vp^i, 0), (0, \vp^j)], (x^k,
x^l)\Big\rangle\nonumber \\
&&\phantom{\langle [(\vp^i,}= \Big\langle (\vp^i, 0) \otimes (0, \vp^j), \D_r\big((x^k,
x^l)\big)\Big\rangle\nonumber\\
&&\phantom{\langle [(\vp^i,}= q\d_{j, m} \big((1-k)\d_{i, k}, 0\big) + \d_{i,1}(0, mq\d_{j, m+k-1} -
l\d_{j, l})-\d_{i,1}(0, -l\d_{j, l})\nonumber\\
&&\phantom{\langle [(\vp^i,}=\Big\langle \d_{j, m} \big((1-i)q\vp^{i}, 0\big) + \d_{i,1}(
mq\vp^{j-m+1},0), (x^k, x^l)\Big\rangle,\nonumber
\end{eqnarray}
\noindent we obtain
\begin{equation}\label{eq4-5}
[(\vp^i, 0), (0, \vp^j)]=\d_{j, m} \big((1-i)q\vp^{i}, 0\big) + \d_{i,1}(
mq\vp^{j-m+1},0).
\end{equation}
From this, we have
\begin{equation}
[(\vp^i, 0), (0, \vp^j)]=\left\{\begin{array}{llll} (mq\vp^{j-m+1},0)
&\mbox{if}\ i=1,\\
\big((1-i)q\vp^{i}, 0\big) &\mbox{if}\ i\neq 1,\, j=m,\\
0 
&\mbox{if }i\neq1,\, j\neq m.
 \end{array}\right.
\nonumber\end{equation}
Similarly, since
\begin{equation}
\begin{array}{lll} \Big\langle [(0, \vp^i), (0, \vp^j)], (x^k,
x^l)\Big\rangle &=&\Big\langle (0, \vp^i)\otimes (0, \vp^j), (x^k, x^l)\cdot
r \Big\rangle \\
&=&(0, -l\d_{i,l})q\d_{j,m} -q\d_{i, m}(0, -l\d_{j, l})\\
&=&\Big\langle \d_{j, m}(0, -iq\vp^i) -\d_{i, m}(0, -jq\vp^j), (x^k,
x^l)\Big\rangle,
\nonumber\end{array}
\end{equation}
\noindent we have
\begin{equation}
[(0, \vp^i), (0, \vp^j)]=\d_{i, m}(0, jq\vp^j) - \d_{j, m}(0,
iq\vp^i)= \left\{\begin{array}{llll} (0, jq\vp^j) &\mbox{if}\ i=m,\, j\neq m,\\
 0 
&\mbox{if}\ i\neq m,\, j\neq m.\nonumber\end{array}\right.
\end{equation}
The third case of (1) holds.

Now we prove case (2). For $(x^k, x^l) \in \LL$ with $k, l\in \Z,$
\begin{equation}
\begin{array}{lll} (x^k, x^l)\cdot r^{\prime} &=& (x^k, x^l)\cdot
\big((x, 0) \otimes (0, qx^m) - (0,qx^m)\otimes (x, 0)\big) \\
&=& \big((1-k)x^k, -lx^l\big)\otimes (0, qx^m) + (x, 0)\otimes (0,
mqx^{k+m-1})\\
&&- (0, mqx^{m+k-1}) \otimes (x, 0) - (0, qx^m)\otimes \big((1-k)x^k,
-lx^l\big),\nonumber\end{array}
\end{equation}
\noindent we obtain
$\langle [(\vp^i,0), (\vp^j, 0)], (x^k, x^l) \rangle= 0.$
Therefore, $[(\vp^i,0), (\vp^j, 0)]=0$,  case (a) is obtained.

Similarly, from
\begin{equation}
\begin{array}{lll} \Big\langle [(\vp^i, 0), (0, \vp^j)], (x^k,
x^l)\Big\rangle &=& \big((1-k)\d_{i, k}, 0\big)q\d_{j, m} +\d_{i,1}(0,
mq\d_{j, k+m-1})\\
&=& \Big\langle \d_{j, m}\big((1-i)q\vp^i, 0\big) +\d_{i,1}( mq\vp^{j-m+1}, 0),
(x^k, x^l) \Big\rangle,\nonumber
\end{array}
\end{equation}
we obtain
\begin{equation}
\begin{array}{lll}  [(\vp^i, 0), (0, \vp^j)]&=&
\d_{j, m}\big((1-i)q\vp^i 0\big) +\d_{i,1}(mq\vp^{j-m+1}, 0)\\
&=&\left\{\begin{array}{llll} (mq\vp^{j-m+1},0) &\mbox{if}\ i=1, \\
\big((1-i)q\vp^i, 0\big) &\mbox{if}\ i\neq 1,\, j=m,\\
0
&\mbox{if}\ i\neq 1,\, j\neq m.
\end{array}\right.\nonumber\end{array}
\end{equation}
From the   relation
\begin{equation}
\begin{array}{lll} \Big\langle [(0,\vp^i), (0, \vp^j)], (x^k,
x^l)\Big\rangle &=& (0, -l\d_{i,l})q\d_{j, m} -q\d_{i,m}(0, -l\d_{j,l})\\
&=& \Big\langle \d_{i, m}(0, jq\vp^j ) - \d_{j, m}(0, iq\vp^i), (x^k,
x^l) \Big\rangle,
\end{array}\nonumber
\end{equation}
it follows that
\begin{equation}
[(0,\vp^i), (0, \vp^j)]=\d_{i, m}(0, jq\vp^j ) - \d_{j, m}(0,
iq\vp^i)=\left\{\begin{array}{llll}(0, jq\vp^j ) &\mbox{if}\ i=m,\, j\neq
m, \\
0 &\mbox{if}\ i\neq m,\, j\neq m.
\end{array}\right.\nonumber\end{equation}
\\[-30pt]\phantom{amd} \
\end{proof}

\begin{theo}\label{Th-4-5} Let $(\LL, \v, \d)$ be the  Lie
bialgebra with $\d$ defined by \eqref{d-defby}.
Then its dual Lie bialgebra
is $(\LL^{\circ}, [\cdot, \cdot], \D)$, where $\LL^{\circ}$ is
determined by \eqref{MaxHVir}
, the cobracket $\D$ is determined by Theorem
$\ref{Th-4-1}$, and the bracket is uniquely determined by
\begin{enumerate}\item[{\rm (1)}] $[(\vp^i, 0), (\vp^j, 0)]=0,$
\item[{\rm (2)}] $[(\vp^i, 0), (0, \vp^j)]=0,$
\item[{\rm (3)}] $[(0, \vp^i), (0, \vp^j)]=\left\{\begin{array}{llll}
(j\a + \g)(\vp^{j+1}, 0)+\b(0, \vp^j) &\mbox{if}\ i=1,\, j\neq 1,\\
0 &\mbox{if}\ i\neq 1,\, j \neq 1.
\end{array}\right.$
\end{enumerate}
\end{theo}
\begin{proof}
For $k,\,l \in \Z,$
\begin{equation}
\begin{array}{lll} \d\big((x^k, x^l)\big) &=& \d\big((x^k, 0) + (0, x^l)\big)\\
&=& \big((k-1)\a + \g\big)\big((0, 1) \otimes (0, x^{k-1}) - (0, x^{k-1}) \otimes
(0, 1)\big)\\
&&+ \b\big((0, 1) \otimes (0, x^l) - (0, x^l)\otimes (0, 1)\big),
\end{array}\nonumber
\end{equation}
we obtain
$\langle [(\vp^i, 0), (\vp^j, 0)], (x^k, x^l)\rangle = \big\langle
(\vp^i, 0)\otimes (\vp^j, 0), \d\big((x^k, x^l)\big)\big\rangle=0.$ Thus
$[(\vp^i, 0), (\vp^j, 0)]=0$.
Similarly, we have $[(\vp^i, 0), (0, \vp^j)]=0.$
Finally, by computation, we have
\begin{eqnarray}&&\Big\langle [(0, \vp^i), (0, \vp^j)], (x^k,
x^l)\Big\rangle\nonumber\\
&&\phantom{\langle [(0,}= \Big\langle (0, \vp^i)\otimes (0, \vp^j), \d\big((x^k,
x^l)\big)\Big\rangle\nonumber \\
&&\phantom{\langle [(0,}=\big((k-1)\a + \g\big)\big(\d_{i,0}(0, \d_{j,k-1})- \d_{j, 0}(0, \d_{i,
k-1})\big)+ \b\big(\d_{i,0}(0, \d_{j, l}) - \d_{j, 0}(0, \d_{i, l})\big)\nonumber\\
&&\phantom{\langle [(0,}=\Big\langle (j\a + \g)\d_{i,0}(\vp^{j+1}, 0)-(i\a + \g) \d_{j, 0}(\vp^{i+1},0)+ \b\big(\d_{i,0}(0, \vp^j) - \d_{j, 0}(0, \vp^i)\big), (x^k,
x^l)\Big\rangle,\nonumber
\end{eqnarray}
\noindent Further, we can get
\begin{equation*}
[(0, \vp^i), (0, \vp^j)]=(j\a + \g)\d_{i,0}(\vp^{j+1}, 0)-(i\a + \g)
\d_{j, 0}(\vp^{i+1},0) + \b\big(\d_{i,0}(0, \vp^j)- \d_{j, 0}(0,
\vp^i)\big).\end{equation*}
From this, we obtain
\begin{equation}
[(0, \vp^i), (0, \vp^j)]=\left\{\begin{array}{llll}
(j\a + \g)(\vp^{j+1}, 0)+\b(0, \vp^j) &\mbox{if}\ i=0,\, j\neq 0,\\
0 &\mbox{if}\ i\neq 0,\, j \neq 0.
 \end{array}\right.\nonumber
\end{equation}
\\[-35pt]\phantom{amd} \
\end{proof}

Obviously, the Lie algebra structures defined in Theorems
\ref{Th-4-2}, \ref{Th-4-3}, \ref{Th-4-4} and \ref{Th-4-5} can be applied to the
the underlining space ${\cal L}=\C[\vp,\vp^{-1}]
\oplus\C[\vp,\vp^{-1}]$. Thus, as by-products, we obtain four new
classes of infinite dimensional Lie algebras $({\cal L}, [\cdot,
\cdot])$. 

\small

\end{CJK*}

\begin{thebibliography}{DWH99}
\def\re#1{\label{#1}\bibitem{#1}}
\small
\parindent=8ex\parskip=1.3pt\baselineskip=1.3pt

\re{A} E. Arbarello, C. De Concini, V. Kac, C. Procesi, Moduli
spaces of curves and representation theory, {\it Comm. Math. Phys.} {\bf117}
(1988),1-36.


\re{B} Y. Billing, Representions of twisted Heisenberg-Virasoro algebra at level zero, {\it Canad. Math. Bull.} {\bf46} (2003),  529-533.

\re{1} R. Block,  Commutative Hopf algebras, Lie coalgebras, and
divided powers, {\it J. Algebra} {\bf 96} (1985), 275--306.

\re{1-2} R. Block, P. Leroux, Generalized dual coalgebras of algebras, with applications to cofree coalgebras,
{\it J. Pure Appl. Algebra} {\bf36} (1985), 15--21.


\re{D}B. Diarra,
On the definition of the dual Lie coalgebra of a Lie algebra, {\it Publ. Mat.} {\bf39} (1995), 349--354.

\re{D1} V. Drinfel'd, Hamlitonian structures on Lie group, Lie
algebras and the geometric meaning of classical Yang-Baxter
equations, {\it Soviet Math. Dokl.} {\bf27}  (1983), 68--71.

\re{2} V. Drinfel'd, Quantum groups, Proceedings ICM (Berkeley 1986),
Providence: Amer Math Soc, 1987, 789--820.

\re{E-D} E. Arbarello, C. De Concini, V. Kac, C. Procesi, Moduli spaces of curves and representation theory,
{\it Comm. Math. Phys.} {\bf177} (1988), 1-36.

\re{F-S} G. Fan, Y. Su, H. Wu,  Loop Heisenberg-Virasoro Lie conformal algebra, {\it J. Math. Phys.} {\bf55} (2014),   123508, 8 pp.

\re{G}G. Griffing,  The dual coalgebra of certain infinite-dimensional Lie algebras, {\it Comm. Algebra} {\bf30} (2002), 5715--5724.

\re{L-C} D. Liu, C. Jiang, Harish-Chandra modules over the twisted heisenberg-Virasoro algebra, {\it J. Math. Phy.} {\bf49} (2008), 012901,13pp.

\re{L-P} D. Liu, Y. Pei, L. Zhu, Lie bialgebra structures on the Heisenberg-Virasoro algebra, {\it J. Algebra} {\bf359} (2012),35-48.

\re{L-W} D. Liu, Y. Wu, L. Zhu,  Whittaker modules over the twisted heisenberg-Virasoro algebra, {\it J. Math. Phy.} {\bf51} (2010), 023524,12pp.


\re{L-Z} R. Lv, K. Zhao, Classification of irreducible weight modules over the twisted Heisenberg-Virasoro algebra,
{\it Comm. Contemp. Math.} {\bf12} (2010), 183--205.

\re{3} S. Majid, Foundations of quantum group theory, Cambridge
University Press 1995.

\re{4} W. Michaelis, A class of infinite dimensional Lie bialgebras
containing the Virasoro algebras, {\it Adv. Math.} {\bf107} (1994), 365--392.

\re{M1} W. Michaelis,
The dual Lie bialgebra of a Lie bialgebra, Modular interfaces (Riverside, CA, 1995), 81--93,
AMS/IP Stud. Adv. Math., 4, Amer. Math. Soc., Providence, RI, 1997.


\re{5}W. Nichols, The structure of the dual Lie coalgebra of
the Witt algebras, {\it J. Pure Appl. Algebra} {\bf68} (1990), 359--364.

\re{6} W. Nichols, On Lie and associative duals, {\it J. Pure Appl.
Algebra} {\bf87} (1993), 313--320.

\re{7} B. Peterson, E. Taft, The Hopf algebra of linearly
recursive sequences, {\it Aequationes Mathematicae} {\bf20} (1980), 1--17,
University Waterloo.

\re{S-C} R. Shen, C. Jiang, The derivation algebra and automorphism group of twisted heisenberg-Virasoro algebra,
{\it Comm. Algebra} {\bf34} (2006), 2574--2558.

\re{8} E. Taft, Witt and Virasoro algebras as Lie bialgebras,
{\it  J. Pure Appl. Algebra} {\bf87} (1993), 301--312.


\re{9} S.-H. Ng, E. Taft, Classification of the Lie bialgebra
structures on the Witt and Virasoro algebras, {\it J. Pure Appl. Algebra}
{\bf151} (2000), 67--88.

\re{10} G. Song, Y. Su, Lie Bialgebras of generalized Witt type,
{\it Science in China A}  {\bf49} (2006), 533--544.

\re{16} G. Song, Y. Su, Dual Lie Bialgebras of  Witt and Virasoro
Types (in Chinese), {\it Sci Sin Math} {\bf 43} (2013), 1093--1102, doi:
10.1360/012013-164.

\re{SS} G. Song, Y. Su, Dual Lie Bialgebra Structures of Poisson
Types, {\it Science in China A} {\bf58} (2015), 1151--1162,
doi: 10.1007/s11425-015-4991-7.

\re{SS-2} G. Song, Y. Su, Dual Lie Bialgebra Structures of Loop and
Current Virasooro Algebras (in Chinese), (to appear in Sci Sin Math,
2017).

\re{S} Y. Su, Low dimensional cohomology of general conformal
algebras $gc_N$, {\it J. Math. Phy.} {\bf54} (2013),  053503.




\re{12} M. Sweedler, Hopf Algebras, W. A. Benjamin, Inc. New
York, 1969.


\re{13} Y. Wu, G. Song, Y. Su, Lie bialgebras of generalized
Virasoro-like type,  {\it Acta Math. Sinica Engl. Ser.} {\bf22} (2006),
1915--1922.

\re{14}B. Xin, G. Song, Y. Su, Hamiltonian type Lie bialgebras,
{\it Science in China A} {\bf50} (2007),
1267--1279.
\end{thebibliography}
\end{document}